\documentclass[a4paper,11pt]{article}
\usepackage{graphics}
\usepackage{latexsym}
\usepackage{cite}
\usepackage{graphicx,psfrag}
\usepackage{amsmath,amssymb,amsthm}
\usepackage[bf, small]{titlesec}
\usepackage{relsize}

\newtheorem{Thm}{Theorem}[section]

\newtheorem{Lem}[Thm]{Lemma}

\newtheorem{Cor}[Thm]{Corollary}

\newtheorem{Def}[Thm]{Definition}

\newcommand{\NN}{\ensuremath{\mathbb{N}}}

\newcommand{\ZZ}{\ensuremath{\mathbb{Z}}}

\newcommand{\BBB}{\ensuremath{\mathcal{B}}}

\title{On the matrix sequence $\{\Gamma(A^m)\}_{m=1}^\infty$ for a Boolean matrix $A$ whose digraph is linearly connected}

\author{
\begin{tabular}{c}
{\sc Jihoon Choi} and {\sc Suh-Ryung KIM}\thanks{Corresponding author}\, \thanks{This work was supported by
the National Research Foundation of Korea(NRF) grant funded by the Korea government(MEST) (No.\ NRF-2010-0009933).}
\\{\footnotesize Department of Mathematics Education,
Seoul National University, Seoul 151-742, Korea}
\end{tabular} }

\begin{document}
\maketitle
\begin{abstract}
In this paper, we extend the results given by Park~{\em et al.}~\cite{ppk} by studying the convergence of the matrix sequence $\{\Gamma(A^m)\}_{m=1}^\infty$ for a matrix $A \in \mathcal{B}_n$ the digraph of which is linearly connected with an arbitrary number of strong components. In the process for generalization, we concretize ideas behind their arguments. We completely characterize $A$ for which $\{\Gamma(A^m)\}_{m=1}^\infty$ converges. Then we find its limit when all of the irreducible diagonal blocks are of order at least two. We go further to characterize  $A$ for which the limit of $\{\Gamma(A^m)\}_{m=1}^\infty$ is  a $J$ block diagonal matrix.   All of these results are derived by studying the $m$-step competition graph of the digraph of $A$.
\end{abstract}



\section{Introduction}
 For the two-element Boolean algebra $\mathcal{B}=\{0,1\}$, $\mathcal{B}_n$ denotes the set of all $n \times n$ matrices over $\mathcal{B}$.  Park~{\em et al.}~\cite{ppk} studied the convergence of $\{\Gamma(A^m)\}_{m=1}^\infty$ for a matrix $A \in \mathcal{B}_n$ whose digraph has at most two strong components where $\Gamma(A)=(\gamma_{ij})$ is defined by \[\gamma_{ij}=\left\{ \begin{array}{ll} 0 & \text{if }i=j; \\
0 & \text{if $i\neq j$ and the inner product of row }i\text{ and row }j\text{ of $A$ is }0; \\
1  &\text{if $i \neq j$ and the inner product of row $i$ and row $j$ of $A$ is not $0$.} \end{array}\right.\]
They noticed that, for a matrix $A \in \mathcal{B}_n$, $\Gamma(A)$ is the adjacency matrix of the competition graph of the digraph of $A$. Given a matrix $A$ in $\mathcal{B}_n$, there exists a unique digraph whose adjacency matrix is $A$. We call such a digraph the {\em digraph of $A$} and denote it by $D(A)$.

Given a digraph $D$, the {\em competition graph} $C(D)$ of $D$
has the same vertex set as $D$ and has an edge between vertices $u$ and $v$
if and only if there exists a common prey of $u$ and $v$ in $D$. If $(u,v)$  is an arc of a digraph $D$,
then we call $v$ a {\em prey} of $u$ (in $D$) and call $u$ a {\em predator} of $v$ (in $D$). A graph $G$ is called the {\em row graph} of a matrix $M$ if the rows of $M$ are the vertices of $G$, and two vertices are adjacent in $G$ if and only if their corresponding rows have a nonzero entry in the same column of $M$. This notion was studied by Greenberg~{\em et al.}~\cite{glm}. As noted in \cite{glm}, the competition graph of a digraph $D$ is the row graph of its adjacency matrix. Thus it can easily be checked that the adjacency matrix of the competition graph of a digraph $D$ is $\Gamma(A)$ where $A$ is the adjacency matrix of $D$.

The notion of competition graph is due to Cohen~\cite{cohen1} and has arisen from ecology. Competition graphs also have applications in coding, radio
transmission, and modeling of complex economic systems.  (See \cite{RayRob} and \cite{Bolyai} for a summary of these applications.)

 The greatest common divisor of all lengths of directed cycles in a nontrivial digraph $D$ is called the \textit{index of imprimitivity of $D$}. A digraph $D$ is said to be {\em primitive} if $D$ is strongly connected and has the index of imprimitivity $1$. Let  $A$ be a matrix in $\mathcal{B}_n$. We call the index of imprimitivity of $D(A)$ the \textit{index of imprimitivity of $A$}. If $D(A)$ is primitive, then we say that $A$ is {\em primitive}. If $D(A)$ is strongly connected, then we say $A$ is {\em irreducible}. For undefined terms in this paper, the reader may refer to \cite{Brualdi}.

 It is well-known that for an irreducible matrix $A$ in $\mathcal{B}_n$, the matrix sequence $\{A^m\}_{m=1}^\infty$ converges if and only if $A$ is primitive.
Yet, a matrix sequence $\{\Gamma(A^m)\}_{m=1}^\infty$ might converge even if the matrix $A$ is not primitive.
For example, the $m$th power of the matrix $A$ given in Figure~\ref{fig1} does not converge as $m$ increases since it is not primitive. However, the sequence $\{\Gamma(A^m)\}_{m=1}^{\infty}$ converges to $A'$
 since $\Gamma(A^m)=A'$ for any positive integer $m$.
\begin{figure}
\begin{center}
{\footnotesize
\[\begin{array}{lllll}
A=\left(\begin{array}{cccc} 0 & 1 & 0 & 1 \\ 0 & 0 & 1 & 0 \\ 1 & 0 & 0 & 0 \\ 0 & 0 & 1 & 0 \end{array} \right)
& A^2=\left(\begin{array}{cccc} 0 & 0 & 1 & 0 \\ 1 & 0 & 0 & 0 \\ 0 & 1 & 0 & 1 \\ 1 & 0 & 0 & 0 \end{array} \right)
& A^3=\left(\begin{array}{cccc} 1 & 0 & 0 & 0 \\ 0 & 1 & 0 & 1 \\ 0 & 0 & 1 & 0 \\ 0 & 1& 0 & 1 \end{array} \right)
& A^4=\left(\begin{array}{cccc} 0 & 1 & 0 & 1 \\ 0 & 0 & 1 & 0 \\ 1 & 0 & 0 & 0 \\ 0 & 0 & 1 & 0 \end{array} \right)
& \cdots
\end{array}\]
\[A'=\left(\begin{array}{cccc} 0 & 0 & 0 & 0 \\ 0 & 0 & 0 & 1 \\ 0 & 0 & 0 & 0 \\ 0 & 1 & 0 & 0 \end{array} \right)\]
\caption{An example given by Park~{\em et al.}~\cite{ppk}. We note that $A$, $A^2$, $A^3$ are all distinct and $A^4=A$. Thus $\{A^m\}_{m=1}^{\infty}$ does not converge. However $\Gamma(A^m)=A'$ for each positive integer $m$.
}
 \label{fig1} }
\end{center}
\end{figure}

Park~{\em et al.}~\cite{ppk} showed that $\{\Gamma(A^m)\}_{m=1}^\infty$ converges as $m$ increases for any irreducible Boolean matrix $A$ and its limit is a block diagonal matrix each of whose blocks consists of all $1$s except in the main diagonal and all $0$s in the main diagonal up to conjugation by simultaneous permutation of rows and columns.  They also completely characterized a matrix $A \in \mathcal{B}_n$  whose digraph  has exactly two strongly connected components (in short, strong components) and for which  $\{\Gamma(A^m)\}_{m=1}^\infty$ converges, and found the limit of $\{\Gamma(A^m)\}_{m=1}^\infty$ in terms of its digraph when it converges. They derived these facts in terms of the competition graph of the digraph of $A$.

 Given a digraph $D$ and a positive integer $m$, a vertex $y$ is an {\em $m$-step prey} of a vertex $x$ if and only if there exists a directed walk from $x$ to $y$ of length $m$. Given a digraph $D$ and a positive integer $m$, the digraph $D^m$ has the same vertex set as $D$ and has an arc $(u,v)$ if and only if $v$ is an $m$-step prey of $u$. It is well-known that a digraph $D$ is primitive if and only if $D^m$ equals the digraph which has all possible arcs for any $m \ge N$ for some positive integer $N$ (we call the smallest such integer $N$ the {\em exponent} of $D$). Motivated by this, Park~{\em et al.}~\cite{ppk} introduced the notion of convergence of $\{G_n\}_{n=1}^{\infty}$ as follows: A graph sequence $\{G_n\}_{n=1}^{\infty}$ (resp.\ digraph sequence) \emph{converges} if there exists a positive integer $N$ such that $G_n$ is equal to $G_N$ for any $n\ge N$. They called the graph $G_N$ the \textit{limit} of the graph sequence (resp.\ digraph sequence). Then they translated their goals described above into competition graph version and showed that $\{C(D^m)\}_{m=1}^\infty$ converges to a graph with only complete components as $m$ increases if $D$ is strongly connected,  completely characterized a digraph $D$ with exactly two strong components for which $\{C(D^m)\}_{m=1}^\infty$ converges, and found the limit of $\{C(D^m)\}_{m=1}^\infty$ when $\{C(D^m)\}_{m=1}^\infty$ converges.

Given a positive integer $m$, the {\em $m$-step competition graph} of a digraph $D$, denoted by $C^m(D)$, has the same vertex set as $D$ and has an edge between vertices $u$ and $v$ if and only if there exists an $m$-step common prey of $u$ and $v$. The notion of $m$-step competition graph is introduced by Cho~{\em et al.}~\cite{firsti} and one of the important variations (see the survey articles by Kim~\cite{Kim93} and Lundgren~\cite{Lundgren89} for the variations which have been defined and studied by many authors since Cohen introduced the notion of competition). Since its introduction, it has been extensively studied (see for example \cite{Belmont,ChoHKim,He,HKim,Ho,PLK,ZC}). Cho~{\em et al.}~\cite{ChoHKim} showed that for any digraph $D$ and a positive integer $m$, $C^m(D)=C(D^m)$. Thus the limit of the  graph sequence $\{C(D^m)\}_{m=1}^\infty$, if it exists, is the same as that of the graph sequence $\{C^m(D)\}_{m=1}^\infty$.  Consequently studying the  graph sequence $\{C(D^m)\}_{m=1}^\infty$ is actually studying the sequence of $m$-step competition graphs of $D$.

In this paper, we extend the results given by Park~{\em et al.}~\cite{ppk} by studying the convergence of $\{\Gamma(A^m)\}_{m=1}^\infty$ for a matrix $A \in \mathcal{B}_n$ satisfying
\begin{equation} \label{form}
PAP^T=\left[\begin{array}{cccccc} A_{11} & A_{12} & O & O &  \cdots & O  \\ O & A_{22} & A_{23} & O & \cdots & O \\ O & O & A_{33} & A_{34} & \cdots & O \\ \vdots & \vdots & \vdots & \ddots & \vdots & \vdots \\ O & O & O & O & \cdots & A_{\eta\eta}
\end{array} \right]
\end{equation}
for a permutation matrix $P$ of order $n$, irreducible matrices $A_{11}$, $A_{22}$, $\ldots$, $A_{\eta \eta}$, and nonzero matrices $A_{12}$, $A_{23}$, $\ldots$, $A_{\eta-1,\eta}$
where if a diagonal block $A_{ii}$ is of order at least two and $\kappa_i$ is the index of imprimitivity of the digraph of $A_{ii}$, then
\begin{equation*}
A_{ii} =
\begin{bmatrix}
	O & A_{12}^{(ii)} & O & \cdots & O & O  \\
	O & O & A_{23}^{(ii)} & \cdots & O & O \\
	O & O & O & \cdots & O & O \\
	\vdots & \vdots & \vdots & \ddots & \vdots & \vdots \\
	O & O & O & \cdots & O & A_{\kappa_i -1, \kappa_{i}}^{(ii)} \\
	A_{\kappa_i,1}^{(ii)} & O & O & \cdots & O & O\\
\end{bmatrix}.
\end{equation*}
We completely characterize $A$ for which $\{\Gamma(A^m)\}_{m=1}^\infty$ converges in Section~2.  Then, in Section~3, we find its limit  when $A_{11}$, $A_{22}$, $\ldots$, $A_{\eta\eta}$ are of order at least two. In this case, the convergence of $\{\Gamma(A^m)\}_{m=1}^\infty$ is guaranteed by one of their results:
\begin{Thm}[\cite{ppk}]\label{thm:nontrivial}
If a digraph $D$ is trivial or each vertex of $D$ has an out-neighbor, then $\{C(D^m)\}_{m=1}^{\infty}$ converges.
\label{cor:strong1}
\end{Thm}
\noindent We go further to generalize one of their results to characterize a matrix $A$ with $A_{11}$, $A_{22}$, $\ldots$, $A_{\eta\eta}$ of order at least two for which the limit of $\{\Gamma(A^m)\}_{m=1}^\infty$ is a $J$ block diagonal matrix. Adopting the notion defined by Park~{\em et al.}~\cite{ppk}, we mean by a {\em $J$ block diagonal} (for short {\em JBD}) matrix a block diagonal matrix each of whose blocks consists of all $1$s except in the main diagonal and all $0$s in the main diagonal up to conjugation by simultaneous permutation of rows and columns.

We derive our results by studying the convergence of $\{C(D^m)\}_{m=1}^{\infty}$ for the digraph $D$ of a matrix $A$ satisfying (\ref{form}). The digraph $D$ is weakly connected and has strong components $D_1$, $\ldots$, $D_\eta$ each of whose arcs goes only from $D_i$ to $D_{i+1}$ for some $i \in \{1,\ldots,\eta-1\}$, which shall be said to be  {\em linearly} connected. We note that a weakly connected digraph with two strong components is linearly connected.

Given a linearly connected digraph $D$ with $\eta$ strong components,  unless otherwise stated, we mean by $D_1$, $\ldots$, $D_\eta$ the strong components of $D$ each of whose arcs goes only from $D_j$ to $D_{j+1}$ for some $j \in \{1,\ldots,\eta-1\}$ and by $D_{i,i+1}$ the subdigraph of $D$ induced by $V(D_i) \cup V(D_{i+1})$ for each $i=1$, $\ldots$, $\eta-1$. We denote by $\kappa(D_i)$ ($\kappa_i$ for short) the index of imprimitivity of $D_i$ and the sets of imprimitivity of $D_i$ by $U_1^{(i)}$, $U_2^{(i)}$, $\ldots$, $U_{\kappa_i}^{(i)}$ for $i=1$, $\ldots$, $\eta$.

In this paper, all the graphs and digraphs are assumed to be simple.

\section{Convergence of $\{\Gamma(A^m)\}_{m=1}^\infty$}
In this section, we completely characterize a matrix $A \in \BBB_n$ in the form given in (\ref{form}) whose digraph is linearly connected and for which $\{\Gamma(A^m)\}_{m=1}^\infty$ converges.

We denote by $\ell(W)$ the length of a walk $W$ in a graph or digraph.
\begin{Lem}\label{lemma2}
Let $D$ be a linearly connected digraph with $\eta$ strong components. Suppose that a vertex  $x$ has an $m$-step prey in a nontrivial component $D_i$ for some positive integer $m$ and $i \in \{1,2,\ldots,\eta\}$. Then $x$ has an $m'$-step common prey in $D_i$ for all $m' \ge m$. Furthermore, every vertex in $D_i$ is a $k$-step prey of $x$ for some positive integer $k$.	
\end{Lem}

\begin{proof}
Let $z \in V(D_i)$ be an $m$-step prey of $x$. Take an integer $m'$ satisfying $m' \ge m$. Since $D_i$ is nontrivial and strongly connected, there exists a closed directed walk of a positive length containing all the vertices in $D_i$. By traversing such a walk, we may find a vertex $z'$ in $D_i$ such that there exists a directed $(z,z')$-walk of length $m'-m$ in $D_i$. Then $z'$ is an $m'$-step prey of $x$.
	
Take any vertex $w \in V(D_i)$. Since $D_i$ is strongly connected, there exists a directed $(z,w)$-walk $W$. Since $z$ is an $m$-step  prey of $x$, $w$ is a $k$-step prey of $x$ for $k = m + \ell(W)$.	
\end{proof}
\noindent The following corollary is immediately true by the above theorem.
\begin{Cor}\label{cor:lemma2}
Let $D$ be a linearly connected digraph with $\eta$ strong components. Then if any two vertices $x$ and $y$ have an $m$-step common prey belonging to a nontrivial strong component for some positive integer $m$, then there exists a positive integer $N$ such that $x$ and $y$ are adjacent in $C(D^m)$ for any integer $m \ge N$.
\end{Cor}

Corollary~\ref{cor:lemma2} implies that, when two vertices $x$ and $y$ have an $\alpha
$-step common prey $z$ in a nontrivial component of a linearly connected digraph $D$ for a positive integer $\alpha$, the adjacency of  $x$ and $y$ in the limit of $\{C(D^m) \}_{m}^\infty$ is determined by not the value of $\alpha$ but the fact that $x$ and $y$ have a `step common prey' $z$ in a nontrivial component. In this context, we sometimes omit `$\alpha$' and just say that $x$ and $y$ have a `step common prey'.

The following lemma shall be frequently quoted in the rest of this paper.
\begin{Lem} [Lemma 3.4.3, \cite{Brualdi}]\label{brualdi1}
Let $D$ be a nontrivial strongly connected digraph, and $U_1$, $U_2$, $\ldots$, $U_{\kappa(D)}$ be the sets of imprimitivity of $D$.
Then there exists a positive integer $N$ such that if $x$ and $y$ are vertices belonging respectively to $U_i$ and $U_j$, then there are directed $(x,y)$-walks of every length $j-i+t\kappa(D)$ with $t \ge N$.\label{Lem:brualdi}
\end{Lem}

For a digraph $D$ and a vertex $v$ of $D$, $N_D^-(v)$ denotes the set of all in-neighbors of $v$.

Given a linearly connected digraph $D$, suppose that $D_p$ is a nontrivial component while $D_{p+1}$ is a trivial component consisting of vertex $v$. Let
\[\Lambda(D):=\left\{i \mid U^{(p)}_i \cap N_D^-(v) \neq \emptyset\right\}=\{k_1,\ldots,k_s\}.\] Then, by Lemma~\ref{brualdi1}, it is easy to check that for each $j \in \ZZ_{\kappa_p}$ and $x \in U_j^{(p)}$, there is a positive integer $N_p$ such that there exist $(x,v)$-walks of lengths $(k_1-j)+1+t\kappa_p$, $\ldots$, $(k_s-j)+1+t\kappa_p$ for every $t \ge N_p$. We put
\[L_{j \to v}:=\{(k_r-j)+1\,({\rm mod}\, \kappa_p) \mid r=1, 2, \ldots, s\}.\]
\noindent(See Figure~\ref{conv} for an illustration.)
In general, for a nonnegative integer $i$, we set
\[i+ L_{j \to v}:=\{(k_r-j)+i+1\,({\rm mod}\, \kappa_p) \mid r=1, 2, \ldots, s\}\]
where $L_{j \to v}=0+L_{j \to v}$.
\begin{figure}\psfrag{A}{$U_1^{(1)}$}\psfrag{B}{{$U_2^{(1)}$}}
\psfrag{D}{{ $U_1^{(2)}$}}\psfrag{E}{{$U_2^{(2)}$}}
\psfrag{F}{{$U_3^{(2)}$}}\psfrag{G}{{$U_4^{(2)}$}}
\psfrag{H}{{$v$}}\psfrag{J}{$w$}\psfrag{L}
{$L_{1\to v}=\{0,1,2\}$}\psfrag{M}{$L_{2\to v}=\{0,1,3\}$}\psfrag{N}{$L_{3\to v}=\{0,2,3\}$}\psfrag{O}
{$L_{4\to v}=\{1,2,3\}$}\psfrag{K}{$D$}
\begin{center}
\includegraphics[height=5.5cm]{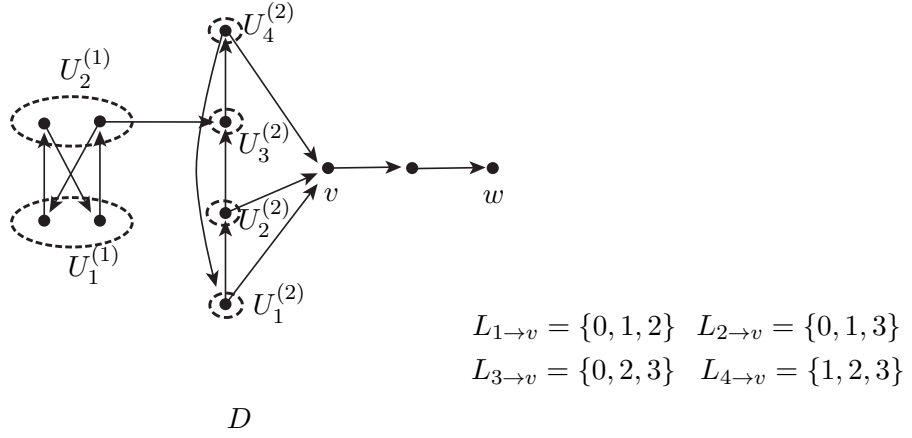}
\end{center}
\caption{It is easy to see that $(L_{i\to v} \cap L_{j\to v})\cup ((1+L_{i\to v}) \cap (1+L_{j\to v}))\cup ((2+L_{i\to v}) \cap (2+L_{j\to v}))=\{0,1,2,3\}$ for any $i$, $j$,  $1 \le i < j \le 4$. Therefore $\{C(D^m)\}_{m=1}^{\infty}$ converges by Theorem~\ref{thm:trivial}. However, if $D'=D-w$, then $(L_{1\to v} \cap L_{2\to v})\cup ((1+L_{1\to v}) \cap (1+L_{2\to v}))=\{0,1,2\}$ and $\{C(D'^m)\}_{m=1}^{\infty}$ diverges by the theorem.}
\label{conv}
\end{figure}
Obviously, if $D_{p+1}$, $\ldots$, $D_{p+\zeta}$ are trivial components of $D$, then, for $i=1$, $\ldots$, $\zeta$,
\begin{align*}
\parbox[h]{13.5cm}{$l \in \left(i+ L_{j_1 \to v} \right) \cap \left(i+ L_{j_2 \to v} \right)$ if and only if the vertex of $D_{p+i+1}$ is an $(l+t\kappa_p)$-step common prey of any vertex in $U_{j_1}^{(p)}$ and any vertex in $U_{j_2}^{(p)}$ for any $t$ greater than or equal to some positive integer $N_{i}$.} \tag{$\ast$}
\end{align*}
\begin{Thm} \label{thm:trivial}
Let $D$ be a linearly connected digraph with exactly $\eta$ strong components. Then  $\{C(D^m)\}_{m=1}^{\infty}$ converges if and only if one of the following is true:
\begin{itemize}
\item[(i)] For each $i=1$, $\ldots$, $\eta$, $D_i$ is trivial.
\item[(ii)] $D_\eta$ is nontrivial.
\item[(iii)] $D_\eta$ is trivial whereas there is a nontrivial strong component in $D$, and if $p$ is the largest index for which $D_p$ is a  nontrivial component, then $\bigcup_{i=0}^{\eta-p-1} ((i+L_{j_1 \to v}) \cap (i+L_{j_2 \to v})) = \emptyset$ or $\ZZ_{\kappa_p}$ for any $j_1$, $j_2$ in $\ZZ_{\kappa_p}$ where $V(D_{p+1})=\{v\}$ and $\Lambda(D)=\left\{i \mid U^{(p)}_i \cap N_D^-(v) \neq \emptyset\right\}=\{k_1,\ldots,k_s\}$ for some integer $s$, $1 \le s \le \kappa_p$.
\end{itemize}
\end{Thm}
\begin{proof}
We show the `if' part first.  If (i) is true, then $C(D^m)$ is an edgeless graph with the vertex set $V(D)$ for any positive integer $m$ and so $\{C(D^m)\}_{m=1}^{\infty}$ converges.  If (ii) is true, then $\{C(D^m)\}_{m=1}^{\infty}$ converges by Theorem~\ref{thm:nontrivial}.

Now we suppose that (iii) is true. Take two vertices $x$ and $y$ of $D$.  If $x$ and $y$ do not have an $\alpha$-step common prey for any positive integer $\alpha$, then $x$ and $y$ are not adjacent in $C(D^m)$ for any positive integer $\alpha$. Now consider the case where $x$ and $y$ have a step common prey. If $x$ or $y$ is a vertex of a trivial component appearing after $D_p$, then $x$ and $y$ cannot have an $\alpha$-step common prey for any positive integer $\alpha$. Thus $x$ and $y$ belong to components appearing before $D_{p+1}$. If they have a step common prey in $D_q$ for $q \le p$, then they have a step common prey in $D_p$ and so there exists an integer $M$ such that $x$ and $y$ are adjacent in $C(D^m)$ for any integer $m \ge M$ by Corollary~\ref{cor:lemma2}. Suppose that $x$ and $y$ have a step common prey only in a trivial component appearing after the component $D_p$ and let $w$ be a step common prey of $x$ and $y$. Then there exist a directed $(x,w)$-walk $W_1$ and a directed $(y,w)$-walk $W_2$ of the same length.  Deleting from $W_1$ and $W_2$ the vertices in trivial components appearing after $D_p$, we obtain a directed $(x,w_1)$-walk and a directed $(y,w_2)$-walk of the same length where $w_1 \in U_{j_1}^{(p)}$ and $w_2 \in U_{j_2}^{(p)}$ for some $j_1,j_2 \in \ZZ_{\kappa_p}$. Then $w_1$ and $w_2$ have a 1-step common prey $v$ and so $1 \in L_{j_1 \to v} \cap L_{j_2 \to v}$. Therefore $L_{j_1 \to v} \cap L_{j_2 \to v}\neq \emptyset$. Then, as we assumed that (iii) is true, $\bigcup_{i=0}^{\eta-p-1} ((i+L_{j_1 \to v}) \cap (i+L_{j_2 \to v})) = \ZZ_{\kappa_p}$. By ($\ast$), there exists a positive integer $N_{\kappa_p}$ such that $x$ and $y$ have an $(l+t\kappa_p)$-step common prey for any $l \in \ZZ_{\kappa_p}$ and $t \ge N_{\kappa_p}$, which implies that $x$ and $y$ are adjacent in $C(D^s)$ for any integer $s$ greater than or equal to $N\kappa_p$.

We show the `only if' part by verifying the contrapositive. Suppose that $D_\eta$ is trivial whereas there is a nontrivial strong component in $D$ and that for the largest index $p$ such that $D_p$ is a  nontrivial component, $\emptyset \varsubsetneq \bigcup_{i=0}^{\eta-p-1} ((i+L_{j_1 \to v}) \cap (i+L_{j_2 \to v})) \varsubsetneq \ZZ_{\kappa_p}$ for some $j_1, j_2 \in \ZZ_{\kappa_p}$. Then $(i_1+L_{j_1 \to v}) \cap (i_1+L_{j_2 \to v}) \neq \emptyset$ for some $i_1 \in \{0,\ldots,\eta-p-1\}$. Therefore, by ($\ast$), for some $l \in \ZZ_{\kappa_p}$, the vertex of $D_{p+i_1+1}$ is a common prey of any vertex in $U_{j_1}^{(p)}$ and any vertex in $U_{j_2}^{(p)}$ in $D^{l+t\kappa_p}$ for any integer $t$ greater than or equal to some positive integer $N_{i_1}$. Thus every vertex in  $U_{j_1}^{(p)}$ and every vertex in $U_{j_2}^{(p)}$ are adjacent in $C(D^{l+t\kappa_p})$ for any integer $t \ge N_{i_1}$. On the other hand, since $\bigcup_{i=0}^{\eta-p-1} ((i+L_{j_1 \to v}) \cap (i+L_{j_2 \to v})) \neq \ZZ_{\kappa_p}$, there is an element $l'$ in $\ZZ_{\kappa_p}$ such that $l' \not\in (i+L_{j_1 \to v}) \cap (i+L_{j_2 \to v})$ for each $i=0$, $\ldots$, $\eta-p-1$. Then, by ($\ast$), for any $i=0$, $\ldots$, $\eta-p-1$ and for any positive integer $N$, there exists an integer $t \ge N$ such that the vertex of $D_{p+i+1}$ is not an $(l'+t\kappa_p)$-step common prey of some vertex in $U_{j_1}^{(p)}$ and some vertex in $U_{j_2}^{(p)}$. That is, for any positive integer $N$, there exists an integer $s \ge N$ such that some vertex in $U_{j_1}^{(p)}$ and some vertex in $U_{j_2}^{(p)}$  are not adjacent in $C(D^{l'+s\kappa_p})$. However, we have shown the existence of $N_{i_1}$ such that every vertex in  $U_{j_1}^{(p)}$ and every vertex in $U_{j_2}^{(p)}$ are adjacent in $C(D^{l+t\kappa_p})$ for any integer $t \ge N_{i_1}$. Hence we can conclude that  $\{C(D^m)\}_{m=1}^{\infty}$ diverges.
\end{proof}

 For a matrix $A$ in the form given in (\ref{form}), $p, q=1$, $\ldots$, $\eta$; $i=1$, $\ldots$, $\kappa_p$;  $j =1$, $\ldots$, $\kappa_{p+1}$, we let
\begin{align*}
\parbox[h]{13.5cm}{$A^{(pq)}_{ij}$ denote the submatrix of $A$ induced by the rows of $A$ intersecting $A_{i,i+1}^{(pp)}$ and the columns of $A$ intersecting $A_{j-1,j}^{(qq)}$, where $A^{(pp)}_{\kappa_p,\kappa_p+1}$ and $A_{0,1}^{(qq)}$ mean $A^{(pp)}_{\kappa_p,1}$ and $A_{\kappa_{q},1}^{(qq)}$, resepctively.}\tag{$\S$}
\end{align*}
\noindent We note that $A_{ij}^{(pq)}$ is a zero matrix if $|p-q| \ge 2$.

Suppose that, for $p \le \eta-1$, $A_{pp}$ is the last diagonal block of order at least two in $A$. Then the set \[\Lambda(D)=\left\{i \mid U^{(p)}_i \cap N_D^-(v) \neq \emptyset\right\}\] used in  Theorem~\ref{thm:trivial} corresponds to the set
\begin{align*}
\Lambda(A):=&  \{i \mid \mbox{The column of $A$ intersecting the trivial block } A_{p+1,p+1} \\
& \ \ \ \  \mbox{and a  row of } A \mbox{ intersecting } A^{(pp)}_{i,i+1}  \mbox{ meet at } 1.\}
\end{align*}
Furthermore $L_{j\to v}$ used in the same theorem corresponds to
\[L_{j\to(p+1)}:=\{(k_r-j)+1\!\!\!\!\!\pmod{\kappa_p} \mid r=1, 2, \ldots, s\}\]
where $\Lambda(A)=\{k_1, k_2, \ldots, k_s\}$.

Now we are ready to translate Theorem~\ref{thm:trivial} into matrix version.
\begin{Cor}
Suppose that $A \in \BBB_n$ is a matrix in the form given in (\ref{form}). Then $\{\Gamma(A^m)\}_{m=1}^\infty$ converges if and only if one of the following holds:
\begin{itemize}
\item[(i)] For each $i=1$, $\ldots$, $\eta$, $A_{ii}$ is of order one.
\item[(ii)] $A_{\eta\eta}$ is of order at least two.
\item[(iii)] $A_{\eta\eta}$ is of order one whereas there is a diagonal block of order at least two of $PAP^T$ for a permutation matrix $P$, and if $p$ is the largest index such that $A_{pp}$ is of order at least two, then $\bigcup_{i=0}^{\eta-p-1} ((i+L_{j_1 \to (p+1)}) \cap (i+L_{j_2 \to (p+1)})) = \emptyset$ or $\ZZ_{\kappa_p}$ for any $j_1$, $j_2$ in $\ZZ_{\kappa_p}$ where $\Lambda(A)=\{k_1,\ldots,k_s\}$ for some integer $s$, $1 \le s \le \kappa_p$. 
\end{itemize}
\end{Cor}

\section{The limit of $\{\Gamma(A^m)\}_{m=1}^\infty$}
In this section, we find the limit of $\{\Gamma(A^m)\}_{m=1}^\infty$ for a matrix in the form given in (\ref{form}) when $A_{11}$, $A_{22}$, $\ldots$, $A_{\eta\eta}$ are of order at least two. The convergence of $\{\Gamma(A^m)\}_{m=1}^\infty$ for such a matrix $A$ is guaranteed by Theorem~\ref{thm:nontrivial} or Theorem~\ref{thm:trivial}.

To characterize the limit of $\{\Gamma(A^m)\}_{m=1}^\infty$ for a matrix $A \in \mathcal{B}_n$ whose digraph  has exactly two strong components and for which  $\{\Gamma(A^m)\}_{m=1}^\infty$ converges, Park~{\em et al.}~\cite{ppk} introduced the notion $B_D$ for a weakly connected digraph $D$ with exactly two strong components $D_1$ and $D_2$ where $D_2$ is nontrivial.
\begin{Def}[\cite{ppk}]\label{def:bipartite}
We take a weakly connected digraph $D$ with exactly two strong components $D_1$ and $D_2$ where $D_2$ is nontrivial. Let $I(D)=\{(k,l) \mid (x,y)\in A(D) \mbox{ for some }x\in U^{(1)}_k, y\in U^{(2)}_l \}$.
Let $B_D=\left(\mathbb{Z}_{\kappa(D_1)},\mathbb{Z}_{\kappa(D_2)}\right)$ be the bipartite graph defined as follows.
If $D_1$ is nontrivial, then $B_D$ has an edge $(i,j)$ if and only if $i \equiv k+1+p \pmod{\kappa(D_1)}$ and $j \equiv l+p \pmod{\kappa(D_2)}$ for some $(k,l) \in I(D)$ and some integer $p$.
If $D_1$ is trivial, then $B_D$ has an edge $(i,j)$ if and only if $j \equiv l-1 \pmod{\kappa(D_2)}$ for some $(1,l) \in I(D)$, which is obtained by substituting $p=-1$ and  $k(D_1)=1$ in the nontrivial case.
\label{def}
\end{Def}
\noindent They described the limit of $\{C(D^m)\}_{m=1}^{\infty}$ using a notion of  `expansion' of $B_D$.
\begin{Def}[\cite{ppk}]\label{def:expansion}
Given a bipartite graph $B=(X,Y)$, we construct a supergraph of $B$ as follows. We write each edge of $B$ in the arc form $(x,y)$ to make clear that $x\in X$ and $y\in Y$. Then we replace  each vertex $v$ with a complete graph $G_v$ (of any size) so that $G_v$ and $G_w$ are vertex-disjoint if $v \neq w$, and join each vertex of $G_x$ and each vertex of $G_y$ whenever either $(x,y)$ is an edge of $B$ or there exists $z \in Y$ such that $(x,z)$ and $(y,z)$ are edges of $B$. We say that the resulting graph is an {\em expansion  of $B$}.
\end{Def}

\begin{Thm}[\cite{ppk}] \label{thm:main2}
Let $D$ be a weakly connected digraph with two strong components $D_1$ and $D_2$ such that no arc goes from $D_2$ to $D_1$, $D_2$ is nontrivial, and $\{C(D^m)\}_{m=1}^{\infty}$  converges.
Then the limit of  $\{C(D^m)\}_{m=1}^{\infty}$ is an expansion of the bipartite graph $B_D$  defined in Definition~\ref{def}.
\end{Thm}
Given a linearly connected digraph $D$ with $\eta$ strong components $D_1$, $\ldots$, $D_\eta$, we recall that $D_{p,p+1}$ denotes the subdigraph of $D$ induced by $V(D_p) \cup V(D_{p+1})$ for each $p=1$, $\ldots$, $\eta-1$.
Noting that  $D_{p,p+1}$ of $D$ is a weakly connected digraph with two strong components, we extend  Definition~\ref{def:bipartite} as follows to take care of the general case given in (\ref{form}).
 \begin{Def}\label{def:partite} Let $D$ be a linearly connected digraph with $\eta$ nontrivial strong components and let
\[I(D_{p,p+1}):=\{(k,l) \mid (x,y)\in A(D) \mbox{ for some }x\in U^{(p)}_k \mbox{ and } y\in U^{(p+1)}_l \}.\] We define the {\em competition skeleton graph} (CS-graph for short)  of $D$ as the  $\eta$-partite graph $PT^{(\eta)}_D = (V_1, V_2, \ldots, V_\eta)$ with
$V_i = \ZZ_{\kappa_i}, 1 \le i \le \eta$ and an edge $(i,j)$ if and only if for $p \in \{1,\ldots, \eta-1\}$, $(i,j) \in E(B_{D_{p,p+1}})$. (See Figure~\ref{cs} for an illustration.)
\end{Def}
\begin{figure}
\psfrag{A}{$U_1^{(1)}$}\psfrag{B}{{$U_2^{(1)}$}}\psfrag{D}{{ $U_1^{(2)}$}}\psfrag{E}{{$U_2^{(2)}$}}\psfrag{F}{{ $U_3^{(2)}$}}\psfrag{G}{{ $U_4^{(2)}$}}\psfrag{H}{{$U_1^{(3)}$}}\psfrag{I}{{ $U_2^{(3)}$}}\psfrag{1}{{ $1$}}\psfrag{2}{{$2$}}\psfrag{3}{{$3$}}\psfrag{4}{{ $4$}}\psfrag{J}{$D$}\psfrag{K}{$PT_D^{(3)}$}
\begin{center}
\includegraphics[height=6.5cm]{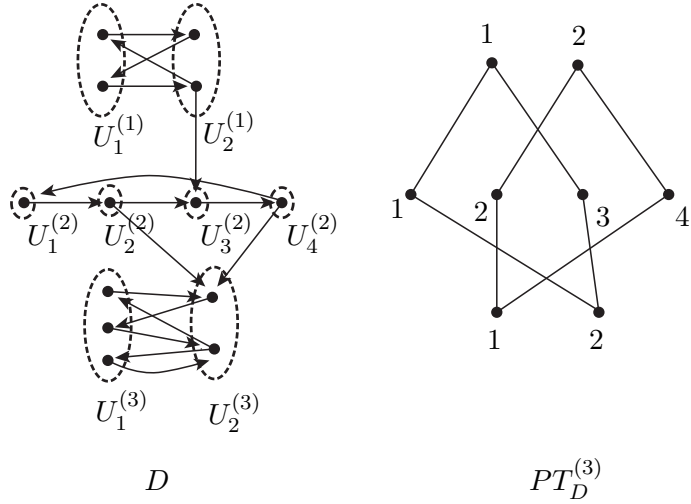}
\end{center}
\caption{A strongly connected digraph $D$ and its CS-graph $PT_D^{(3)}$.}
\label{cs}
\end{figure}
\noindent In the following, for convenience sake, we write $B_{p,p+1}$ for $B_{D_{p,p+1}}$. If $\{i,j\}$ is an edge of $B_{p,p+1}$ where $i \in \ZZ_{\kappa_p}$ and $j \in \ZZ_{\kappa_{p+1}}$, we denote it by $(i,j)$ instead of $\{i,j\}$.

 We state a few useful lemmas.
\begin{Lem}\label{lemmanew}
Let $D$ be a weakly connected digraph with exactly two strong components $D_1$ and $D_2$ both of which are nontrivial.  Then the following are equivalent:
\begin{itemize}
\item[(a)] There exists a directed $(u,v)$-walk of length $2s\kappa_1\kappa_2$ for some vertices $u\in U^{(1)}_i$ and $v\in U^{(2)}_j$ and for some positive integer $s$.
\item[(b)] For any vertices $x \in U^{(1)}_i$ and $y \in U^{(2)}_j$, there exists a positive integer $N(x,y)$ such that, for any $t \ge N(x,y)$, there exists a directed $(x,y)$-walk of length $2t\kappa_1 \kappa_2$.
\end{itemize}
\end{Lem}
\begin{proof} To show (a) implies (b), suppose that there exists a directed $(u,v)$-walk of length $2s\kappa_1\kappa_2$ for some vertices $u\in U^{(1)}_i$ and $v\in U^{(2)}_j$ and for some positive integer $s$. Take any vertices $x \in U^{(1)}_i$ and $y \in U^{(2)}_j$. Then, by Lemma~\ref{brualdi1}, there exist positive integers $N_1$ and $N_2$ such that there are directed $(x,u)$-walks in $D_1$ of every length $s_1\kappa_1$ with $s_1 \ge N_1$ and there are directed $(v,y)$-walks in $D_2$ of every length $s_2\kappa_2$ with $s_2 \ge N_2$.
Let $N = s + N_1 + N_2$. Fix any $t \ge N$ and choose two positive integers $s_1 \ge N_1, s_2 \ge N_2$ satisfying $t = s + s_1 + s_2$. Then  there exist a directed $(x,u)$-walk $Q$ of length $2s_1\kappa_1 \kappa_2$ and a directed $(v,y)$-walk $R$ of length $2s_2\kappa_1 \kappa_2$. Thus we obtain a directed $(x,y)$-walk $QPR$ of length $2(s + s_1 + s_2)\kappa_1 \kappa_2 = 2t\kappa_1 \kappa_2$.

The statement (a) immediately follows from (b) by taking any vertices $u\in U^{(1)}_i$ and $v\in U^{(2)}_j$ and $s=N(u,v)$.
\end{proof}
\noindent By Lemma~\ref{lemmanew}, the following two lemmas are logically equivalent:
\begin{Lem}[\cite{ppk}] \label{lem:claim 1}
Let $D$ be a weakly connected digraph with exactly two strong components $D_1$ and $D_2$ both of which are nontrivial.  Then $(i,j)$ is an edge of $B_D$ if and only if there exists a directed $(u,v)$-walk of length $2s\kappa_1\kappa_2$ for some vertices $u\in U^{(1)}_i$ and $v\in U^{(2)}_j$ and for some positive integer $s$.
\end{Lem}
\begin{Lem} \label{lemma1}
Let $D$ be a weakly connected digraph with exactly two strong components $D_1$ and $D_2$ both of which are nontrivial.  Then $(i,j)$ is an edge of $B_D$ if and only if, for any vertices $x \in U^{(1)}_i$ and $y \in U^{(2)}_j$, there exists a positive integer $N(x,y)$ such that, for any $t \ge N(x,y)$, there exists a directed $(x,y)$-walk of length $2t\kappa_1 \kappa_2$.
\end{Lem}

We shall generalize Lemma~\ref{lem:claim 1}. To do so, we need the following lemma.
\begin{Lem}\label{new}
Let $D$ be a linearly connected digraph with only nontrivial strong components as many as $\eta \ge 3$. Let $x \in U_i^{(p)}$ and $z \in U_k^{(q)}$ for positive integers $p$ and $q$ satisfying $p+2 \le q \le \eta$. If there exists a directed $(x,z)$-walk of length $s\kappa_{p+1}\kappa_{p+2}\cdots\kappa_q$ for some positive integer $s$, then, for some $j \in \ZZ_{\kappa_{p+1}}$ and for some  $y \in U_{j}^{(p+1)}$, there exist
\begin{itemize}
\item[(i)]  a directed $(x,y)$-walk of length $k\kappa_{p+1}$ for each $k$ greater than or equal to some positive integer $K$, in particular, a directed $(x,y)$-walk of length $2K\kappa_{p}\kappa_{p+1}$;
\item[(ii)] a directed $(y,z)$-walk of length $k'\kappa_{p+1}$ for each $k'$ greater than or equal to some positive integer $K'$, in particular,  directed $(y,z)$-walks of lengths $2K'\kappa_{p+1}\kappa_{p+2}$ and $K'\kappa_{p+1}\kappa_{p+2}\cdots\kappa_q$, respectively.
\end{itemize}
\end{Lem}

\begin{proof}
Let $W$ be a directed $(x,z)$-walk of length $s\kappa_{p+1}\kappa_{p+2}\cdots\kappa_q$. Since $D$ is linearly connected, there is a vertex $y_1$ on $W$ belonging to $D_{p+1}$. Then $y_1 \in U_{j_1}^{(p+1)}$ for some $j_1 \in \ZZ_{\kappa_{p+1}}$.  We denote by $W_1$ and $W_2$ the $(x,y_1)$-section of $W$ and $(y_1,z)$-section of $W$, respectively, and then, by $j$ the element in $\ZZ_{\kappa_{p+1}}$ satisfying
\[j \equiv j_1-\ell(W_1) \pmod{\kappa_{p+1}},\]
or
\begin{equation} \label{eq:new}
\ell(W_1) = j_1 - j+ t\kappa_{p+1}
\end{equation}  for some integer $t$.

Now we take a vertex in $U_{j}^{(p+1)}$ and call it  $y$.
 By Lemma~\ref{brualdi1}, there exists $L \in \NN$ such that, for any integer $k \ge L$, there exists a directed $(y_1,y)$-walk of $(j-j_1)+k\kappa_{p+1}$ and so a directed $(x,y)$-walk of length $\ell(W_1)+(j-j_1)+k\kappa_{p+1}$, which equals $(t+k)\kappa_{p+1}$ by (\ref{eq:new}). We let $K= \max\{ L, L+t \}$ and (i) follows.

As $W$ is decomposed into $W_1$ and $W_2$,
\[s\kappa_{p+1}\kappa_{p+2}\cdots \kappa_q = \ell(W_1) +\ell(W_2)
\]
By Lemma~\ref{brualdi1} again, there exists $L' \in \NN$ such that, for each integer $k' \ge L'$, there exists a directed $(y,y_1)$-walk of length $(j_1-j) + k'\kappa_{p+1}$ and so a  directed $(y,z)$-walk of  length $(j_1- j)+ k'\kappa_{p+1} + \ell(W_2)$. Then for each integer $k' \ge L'$,
\begin{align*}
	&(j_1- j) + k'\kappa_{p+1} + \ell(W_2) \\
	= \ &(j_1- j) + k'\kappa_{p+1} + (s\kappa_{p+1}\kappa_{p+2}\cdots \kappa_q - \ell(W_1)) \\
= \ & k'\kappa_{p+1} + (s\kappa_{p+1}\kappa_{p+2}\cdots \kappa_q - t\kappa_{p+1}) \\
	= \ & (k' + s\kappa_{p+2} \cdots \kappa_q - t)\kappa_{p+1}.
\end{align*}
We let $K'=\max\{L',L' + s\kappa_{p+2} \cdots \kappa_q - t\}$, which satisfies (ii).
%
\end{proof}
\begin{Cor}\label{lemma4}
Let $D$ be a linearly connected digraph with only nontrivial strong components as many as $\eta \ge 3$ and let $x \in U_i^{(p)}$ and $z \in U_k^{(p+2)}$ for a positive integers $p$ satisfying $p+2 \le \eta$. If there exists a directed $(x,z)$-walk of length $s\kappa_{p+1} \kappa_{p+2}$ for some positive integer $s$, then
there exists $j \in \ZZ_{\kappa_{p+1}}$ such that $(i,j) \in E(B_{p,p+1})$ and $(j,k) \in E(B_{p+1,p+2})$.
\end{Cor}
\begin{proof} By the hypothesis, there exists $y$ in $D_{p+1}$ satisfying (i) and (ii) of Lemma~\ref{new}. Then, by Lemma~\ref{lem:claim 1}, (i) and (ii) imply $(i,j) \in E(B_{p,p+1})$ and $(j,k) \in E(B_{p+1,p+2})$, respectively, which is the desired conclusion.
\end{proof}
\noindent The following lemma is a generalization of Lemma~\ref{lem:claim 1} and Corollary~\ref{lemma4}.
\begin{Lem}\label{lemma5}
Let $D$ be a linearly connected digraph with only nontrivial strong components as many as $\eta \ge 3$. Let $x \in U_i^{(p)}$ and $z \in U_k^{(q)}$ for positive integers $p$ and $q$ satisfying $p+2 \le q \le \eta$. If there exists a directed $(x,z)$-walk of length $s\kappa_{p+1}\kappa_{p+2}\cdots\kappa_q$ for some positive integer $s$,  then there exists an $(i,k)$-path in the CS-graph of $D$ defined in Definition~\ref{def:partite}.
\end{Lem}


\begin{proof}
We proceed by induction on $m=q-p+1$. The statement holds for $m=3$ by Corollary~\ref{lemma4}. Suppose that the statement holds for $m-1$ ($m \ge 4$). Let $W$ be a directed $(x,z)$-walk of length $s\kappa_{p+1}\kappa_{p+2}\cdots\kappa_q$  for some $s \in \NN$. Then, by Lemma~\ref{new}, for some $j \in \ZZ_{\kappa_{p+1}}$ and for a vertex $y \in U^{(p+1)}_j$, there exist a directed $(x,y)$-walk of length $2K\kappa_p\kappa_{p+1}$ for some $K\in\NN$ and a directed $(y,z)$-walk of length $K'\kappa_{p+1}\kappa_{p+2}\cdots\kappa_q$ for some $K' \in \NN$.  Since there exists a directed $(x,y)$-walk of length $2K\kappa_p\kappa_{p+1}$, by Lemma~\ref{lem:claim 1}, $(i,j) \in E(B_{p,p+1}) \subset E(PT^{(\eta)}_D)$.

Since there is a  directed $(y,z)$-walk of length $K'\kappa_{p+1}\kappa_{p+2}\cdots\kappa_q$, by the induction hypothesis, there exists a $(j,k)$-path in $PT^{(\eta)}_D$, which, together with the edge $(i,j)$, results in an $(i,k)$-path in $PT^{(\eta)}_D$.
\end{proof}

The next theorem plays a key role to describe the limit of  $\{C(D^m)\}_{m=1}^{\infty}$.

\begin{Thm} \label{thm:main}
Let $D$ be a linearly connected digraph with only nontrivial strong components as many as $\eta \ge 3$. Suppose $x \in U_i^{(p)}$ and $y \in U_j^{(q)}$ for integers $p$, $q$, $i$, $j$ with $1 \le p \le q \le \eta$, $i \in \ZZ_{\kappa_p}$, $j\in \ZZ_{\kappa_q}$, respectively. Then, for any integer $r \ge \max\{p+1,q\}$, $x$ and $y$ have a step common prey in $D_r$ if and only if there exist an $(i,k)$-path and a $(j,k)$-path in the CS-graph  of $D$ for some $k \in \ZZ_{\kappa_r}$ satisfying the property that $k=j$ if and only if $q=r$.
\end{Thm}

\begin{proof}
To show the `if' part, suppose that for an integer $r \ge \max\{p+1,q\}$ and an element $k \in \ZZ_{\kappa_r}$, there exist an $(i,k)$-path $P$ and a $(j,k)$-path $Q$ where $k=j$ if and only if  $q=r$. Let
\begin{align*}
P: & i =: i_0 \rightarrow i_1 \rightarrow \cdots \rightarrow i_{\ell(P)-1} \rightarrow i_{\ell(P)} := k, \\
Q: & j =: j_0 \rightarrow j_1 \rightarrow \cdots \rightarrow j_{\ell(Q)-1} \rightarrow j_{\ell(Q)} := k.
\end{align*}
Since $r \ge p+1$, by the definition of a CS-graph, we may take a vertex $x_t \in U_{i_t}^{(p+t)}$, especially $x_0=x$ and, in case $q=r$, $x_{\ell(P)}=y$, for each $t=0$, $1$, $\ldots$, $\ell(P)$. We set $z:=x_{\ell(P)}$. We will show that $z$ is an $m$-step common prey of $x$ and $y$ for some $m \in \NN$. For simplicity, let $\lambda = \kappa_1 \kappa_2 \cdots \kappa_\eta$. By Lemma~\ref{lemma1}, there exists (a sufficiently large integer) $N_1 \in \NN$ such that, for any $s \ge N_1$, there exists a directed $(x_t, x_{t+1})$-walk $W_{t+1}(s)$ of length $2s\lambda$ for each $t = 0,1,\ldots,\ell(P)-1$. For any integer $s \ge N_1$, the concatenation of $W_1(s)$, $\ldots$, $W_{\ell(P)}(s)$ results in a directed $(x,z)$-walk of length $2s\ell(P)\lambda$. Suppose $r=q$. Then $k=j$ and $z=y$ by the assumption. By Lemma~\ref{brualdi1}, there exists a positive integer $N_2$ such that, for any integer $s \ge N_2$, there exists a directed $(y,y)$-walk of length $2s\ell(P)\lambda$. If we let $m=2\ell(P)\lambda \max\{N_1,N_2\}$, then $y$ is an $m$-step common prey of $x$ and $y$.  Now suppose that $r \ge q+1$. Then we may take a vertex $y_t \in U_{j_t}^{(q+t)}$, especially $y_0=y$ and $y_{\ell(Q)}=z$, for $t = 0,\ldots,\ell(Q)$.
 By Lemma~\ref{lemma1}, there exists (a sufficiently large integer) $N_3 \in \NN$ such that, for any integer $s \ge N_3$, there exists a directed $(y_t, y_{t+1})$-walk $W'_{t+1}(s)$ of length $2s\lambda$ for each $t = 0,1,\ldots,\ell(Q)-1$. For any integer $s \ge N_3$, the concatenation of $W'_1(s)$, $\ldots$, $W'_{\ell(Q)}(s)$ results in a directed $(y,z)$-walk of length $2s\ell(Q)\lambda$. As there exists a directed  $(x,z)$-walk of length $2s\ell(P)\lambda$ for any integer $s \ge N_1$, $z$ is an $m$-step common prey of $x$ and $y$ for $m:=2\ell(P)\ell(Q)\lambda\max\{N_1,N_3\}$.

To show the `only if' part, suppose that $x$ and $y$ have an $m$-step common prey  in $D_r$ for some $r \ge \max\{p+1,q\}$ and $m \in \NN$.    By Lemma~\ref{lemma2}, $x$ and $y$ have an $2s\kappa_{p+1} \kappa_{p+2} \cdots \kappa_q$-step common prey $z$ in $D_r$ for some $s \in \NN$ satisfying $s\kappa_{p+1} \kappa_{p+2} \cdots \kappa_q \ge m$. Since $z$ belongs to $D_r$, $z \in U_k^{(r)}$ for some $k \in \ZZ_{\kappa_r}$. If $r \ge p+2$, then there is an $(i,k)$-path in $PT_D^{(\eta)}$ by Lemma~\ref{lemma5} and if $r=p+1$, then there is an edge $(i,k)$ in $B_{p,p+1}$ by Lemma~\ref{lem:claim 1}, which is an $(i,k)$-path in $PT_D^{(\eta)}$. Thus we have shown that there is an $(i,k)$-path in $PT_D^{(\eta)}$.

If $r \ge q+2$, then, by Lemma~\ref{lemma5}, there is a $(j,k)$-path in $PT_D^{(\eta)}$. If $r= q+1$, then there is an edge $(j,k)$ in $B_{q,q+1}$ by Lemma~\ref{lem:claim 1}, which is a $(j,k)$-path in $PT_D^{(\eta)}$.

Now suppose that $r=q$. Since $z$ is a step common prey of $x$ and $y$, there exist a directed $(x,z)$-walk $W_1$ and a directed $(y,z)$-walk $W_2$ of the same length. Since $y$ and $z$ belong to the same strong component, there exists a directed $(z,y)$-walk $W_3$.  Then $W_1W_3$ and $W_2W_3$ are a directed $(x,y)$-walk and a directed $(y,y)$-walk, respectively, of the same length, which implies that $y$ is a step common prey of $x$ and $y$. In particular, by Lemma~\ref{brualdi1}, there exists a positive integer $N$ such that, for any $s \ge N$, $y$ is a $2s\kappa_{p+1} \kappa_{p+2} \cdots \kappa_q$-step common prey of $x$ and $y$. Then there is an $(i,j)$-path in $PT_D^{(\eta)}$ by Lemma~\ref{lem:claim 1} or Lemma~\ref{lemma5} depending upon whether $q=p+1$ or $q \ge p+2$. We take the $(j,j)$-path for $j \in \ZZ_{\kappa_q}$, which is trivial, and set $k:=j$.

If $k=j$, then $q=r$ as $j \in \ZZ_{k_q}$ and $k \in \ZZ_{k_r}$ and so the theorem follows.
\end{proof}

Based on the observations which we have made so far, we generalize the notion of expansion given in Definition~\ref{def:expansion} as follows.
\begin{Def}\label{def:expansion3}
Let $G=(V_1,\ldots,V_\eta)$ be an $\eta$-partite graph without edges between $V_i$ and $V_j$ if $|i-j| \ge 2$. From $G$, we construct a supergraph of $G$ as follows (see Figure~\ref{expnsn} for an illustration):

(Step 1) We replace each vertex $v$ with a complete graph $G_v$ (of any size) so that $G_v$ and $G_w$ are vertex-disjoint if $v \neq w$.

(Step 2) For each $x \in V_i$, $y \in V_j$ where $i \le j$, we join each vertex of $G_x$ and each vertex of $G_y$ if there is an integer $k \ge j$ such that if $j \ge k+1$, then there are an $(x,z)$-path and a $(y,z)$-path in $G$ for some $z \in V_k$ and if $j=k$, then there is an $(x,y)$-path.

\begin{figure}
\psfrag{1}{$x_1$}\psfrag{2}{$x_2$}\psfrag{A}{$G_{x_1}$}\psfrag{B}{$G_{x_2}$}
\psfrag{3}{$y_1$}\psfrag{4}{$y_2$}\psfrag{5}{$y_3$}\psfrag{6}{$y_4$}
\psfrag{D}{$G_{y_1}$}\psfrag{E}{$G_{y_2}$}\psfrag{F}{$G_{y_3}$}\psfrag{G}{$G_{y_4}$}
\psfrag{7}{$z_1$}\psfrag{8}{$z_2$}\psfrag{H}{$G_{z_1}$}\psfrag{I}{$G_{z_2}$}
\psfrag{K}{$G$}\psfrag{J}{$G^*$}
\begin{center}
\includegraphics[height=6.5cm]{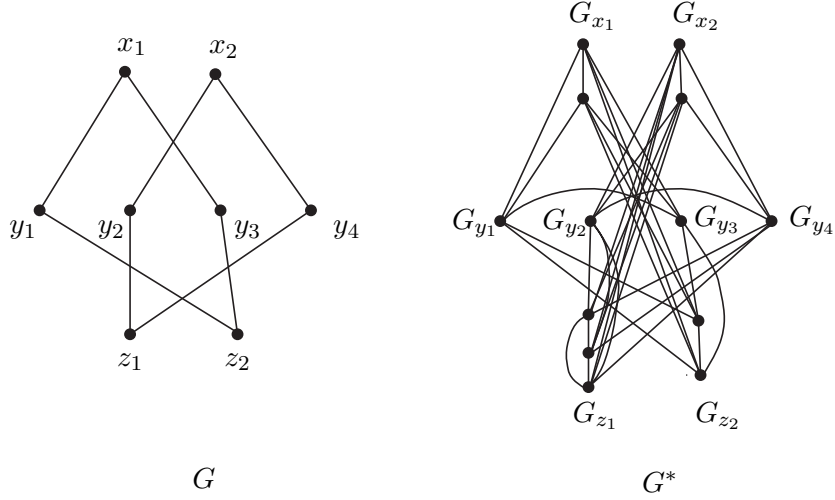}
\end{center}
\caption{A graph $G$ and an expansion $G^*$ of $G$.}
\label{expnsn}
\end{figure}
We say that any graph resulting from this construction is an {\em expansion of $G$}.
\end{Def}
\noindent We note that an expansion of an $\eta$-partite graph $G$ is uniquely determined by the vertex sets of complete graphs replacing the vertices of $G$.

In order to describe the limit of the graph sequence $\{C(D^m)\}_{m=1}^\infty$ for a linearly connected digraph $D$ with only nontrivial strong components, we need one more lemma in the following.
\begin{Lem}[\cite{ppk}] \label{lem:subgraph}
Let $D$ be a weakly connected digraph with exactly two strong components $D_1$ and $D_2$. Then there exists an integer $M$ such that $C(D^m)$ contains  complete graphs whose vertex sets are $U_1^{(1)}$, $\ldots$, $U_{\kappa(D_1)}^{(1)}$, $U_1^{(2)}$, $\ldots$, $U_{\kappa(D_2)}^{(2)}$, respectively, as subgraphs for $m\ge M$.
\end{Lem}

Now we present the theorem stating what the limit of the sequence $\{C(D^m)\}_{m=1}^\infty$ is for a linearly connected digraph $D$ with only nontrivial strong components.
\begin{Thm}\label{thm:main3}
Let $D$ be a linearly connected digraph with only nontrivial strong components as many as $\eta \ge 3$. Then the limit graph of the sequence $\{C(D^m)\}_{m=1}^\infty$ is an expansion of the CS-graph of $D$.
\end{Thm}

\begin{proof}
Denote by $G$ the limit graph of the sequence $\{C(D^m)\}_{m=1}^\infty$. We take the expansion $G^*$ of $PT^{(\eta)}_D$ obtained by replacing the vertex $i$ in $\ZZ_{\kappa_j}$ with the complete graph whose vertex set is the set of imprimitivity $U_i^{(j)}$ of $D_j$ for $i\in \ZZ_{\kappa_j}$ and  $j\in \{1, \ldots, \eta \}$. Obviously $V(G)=V(G^*)$. In the following, we show that $E(G)=E(G^*)$.

Now take two vertices $x$ and $y$ of $G$. Then  $x \in U_i^{(p)}$ and $y \in U_j^{(q)}$ for some $p$, $q$, $i$, $j$ satisfying $1 \le p, q \le \eta$, $i \in \ZZ_{\kappa_p}$, $j \in \ZZ_{\kappa_q}$, respectively. Without loss of generality, we may assume $p \le q$. First, suppose that $x$ and $y$ are adjacent in $G$. Then $x$ and $y$ have a step common prey in $D_r$ for some integer $r \ge q$. If $p=r$, then $q=r$ and so, by one of well-known properties of sets of imprimitivity, $i=j$. Then, by the definition of expansion, $x$ and $y$ are adjacent in $G^*$. Now we assume that $p \le r+1$.  Then, by Theorem~\ref{thm:main}, there exists $k \in \ZZ_{\kappa_r}$ such that there exist an $(i,k)$-path and a $(j,k)$-path in  $PT^{(\eta)}_D$. By the definition of expansion, all the vertices in $U_i^{(p)}$ and all the vertices in $U_j^{(q)}$ are adjacent and so $x$ and $y$ are adjacent in $G^*$. Hence we have shown that $E(G) \subset E(G^*)$.

To show $E(G^*) \subset E(G)$, suppose that vertices $u$ and $v$ are adjacent in $G^*$.  Then, by the definition of expansion, one of the following holds: \begin{itemize}
\item[(i)] $u$ and $v$ belong to the same set of imprimitivity, that is, $\{u,v\} \subset U_j^{(i)}$ for some $i \in \{1, \ldots, \eta \}$ and $j \in\ZZ_{\kappa_i}$;
\item[(ii)] $u$ and $v$ belong to different sets of imprimitivity, i.e., $u \in U_i^{(p)}$, $v \in U_j^{(q)}$ for some integers $p$, $q$ satisfying $1 \le p < q \le \eta$, and either there exist an $(i,k)$-path and a $(j,k)$-path in $PT^{(\eta)}_D$ for some integer $r$, $q+1 \le r \le \eta$, and $k \in \ZZ_{\kappa_r}$ or  there is an $(i,j)$-path.
\end{itemize}
Noting that the $m$-step competition graph of $D_{i,i+1}$ is the subgraph of the $m$-step competition graph of $D$ for any $m \in \NN$ and $i \in \{1, \ldots, \eta-1 \}$, we can conclude that $u$ and $v$ are adjacent in $G$ if (i) is true by  Lemma~\ref{lem:subgraph}. Consider the case (ii). Then, by Theorem~\ref{thm:main}, $u$ and $v$ have a step common prey. Thus, by Corollary~\ref{cor:lemma2}, $u$ and $v$ are adjacent in $G$.
\end{proof}

We consider a matrix $A$ in the form given in (\ref{form}) where $A_{ii}$ has order at least two for each $i=1$, $2$, $\ldots$, $\eta$. Then for a block $A_{i,j}^{(p,p+1)}$ of $A_{p,p+1}$ defined in ($\S$) for $p=1$, $\ldots$, $\eta-1$, it is easy to see that $A^{(p,p+1)}_{ij} \neq O$ if and only if $(i,j) \in I(D_{p,p+1})$ where $D$ is the digraph of $A$. Now Theorem~\ref{thm:main3} may be translated into matrix version in the following way:

\begin{Cor}
Let $A \in \BBB_n$ be a matrix in the form given in (\ref{form}) where $A_{ii}$ has order at least two for each $i=1$, $2$, $\ldots$, $\eta$.
\setlength{\arraycolsep}{10pt} 
Then $\{ \Gamma(A^m) \}_{m=1}^\infty$ converges to a (symmetric) matrix $A'$ such that
\begin{equation*}
PA'P^T =
\begin{bmatrix}
	C_{11} & C_{12} & C_{13} & \cdots & C_{1,\eta-1} & C_{1,\eta} \\
	C_{21} & C_{22} & C_{23} & \cdots & C_{2,\eta-1} & C_{2,\eta} \\
	C_{31} & C_{32} & C_{33} & \cdots & C_{3,\eta-1} & C_{3,\eta} \\
	\vdots & \vdots & \vdots & \ddots & \vdots & \vdots \\
	C_{\eta-1,1} & C_{\eta-1,2} & C_{\eta-1,3} & \cdots & C_{\eta-1,\eta-1} & C_{\eta-1,\eta}\\
	C_{\eta,1} & C_{\eta,2} & C_{\eta,3} & \cdots & C_{\eta,\eta-1} & C_{\eta\eta}\\
\end{bmatrix}
\end{equation*}
where, for each $p$, $q = 1, \ldots, \eta$,  $C_{pp}$ has the same order as that of $A_{pp}$, $C_{pq}=C_{qp}^{T}$, and for the block $C_{ij}^{(pq)}$ of $C_{pq}$, which is of the same order as $A_{ij}^{(pq)}$ and located at the position in $PA'P^T$ where $A_{ij}^{(pq)}$ is located in $PAP^T$,
\[C_{ij}^{(pq)}= \begin{cases} J^* & \text{if (P1) or (P2) is satisfied;} \\
O & \text{otherwise.} \end{cases}\]
\begin{itemize}
\item[(P1)] $p=q$ and $i=j$;
\item[(P2)] $p \neq q$ or $i \neq j$, and either
 \begin{itemize}
 \item there exist positive integers $t_1$ and $t_2$ satisfying  $p+t_1=q+t_2$ such that, for each $r=0$, $\ldots$, $t_1$ and $s=0$, $\ldots$, $t_2$, there exist some positive integers $i_r$ ($i_0=i$), $j_s$ ($j_0=j$), some integers $k_r$, $l_r$, $m_s$, $n_s$, and nonnegative integers $a_r$, $b_s$  for which   $A_{k_rl_r}^{(p+r,p+r+1)} \neq O$,  $A_{m_sn_s}^{(q+s,q+s+1)} \neq O$ and  which satisfy $i_r \equiv k_r+a_r+1 \pmod{\kappa_{p+r}}$, $i_{r+1} \equiv l_r+a_r \pmod{\kappa_{p+r+1}}$, $j_s \equiv m_s+b_s+1 \pmod{\kappa_{q+s}}$, $j_{s+1} \equiv n_s+b_s \pmod{\kappa_{q+s+1}}$,   or
\item $p \neq q$ and, if $p<q$, then there exists a positive integer $t$ satisfying $p+t=q$ such that for each $r=0$, $\ldots$, $t$, there exist some positive integer $i_r$ ($i_0=i$, $i_{q}=j$), some integers $k_r$, $l_r$, and some nonnegative integers $a_r$ for which $A_{k_rl_r}^{(p+r,p+r+1)} \neq O$ and which satisfy $i_r \equiv k_r+a_r+1 \pmod{\kappa_{p+r}}$, $i_{r+1} \equiv l_r+a_r \pmod{\kappa_{p+r+1}}$; if $q < p$, then $C_{pq}^T=J^*$.
\end{itemize}
\end{itemize}
 (We mean by $O$ a zero matrix of any size and by $J^*$ a square matrix of any order such that all the diagonal entries are 0 and all the off-diagonal entries are 1.)
\end{Cor}

Park~{\em et al.}~\cite{ppk} characterized a digraph $D$ for which $\{C(D^m)\}_{m=1}^{\infty}$ converges to a union of complete subgraphs as follows:
\begin{Thm}[\cite{ppk}] Let $D$ be a weakly connected digraph with exactly two strong components $D_1$ and $D_2$ both of which are nontrivial and without arc from $D_2$ to $D_1$. Suppose that $\{C(D^m)\}_{m=1}^{\infty}$ converges to a graph $G$.
Then $G$ is the union of  complete subgraphs if and only if $\kappa_2$ divides $\kappa_1$ and $i-j\equiv i'-j'\pmod{\kappa_2}$ for any $({i},{j})$, $({i'},{j'})\in I(D)$.
\label{thm:simple}
\end{Thm}
In the following, we present a generalization of the above theorem. To do so, we need the following lemma:
\begin{Lem}
Let $D$ be a linearly connected digraph with only nontrivial strong components $D_1$, $D_2$, $\ldots$, $D_\eta$ ($\eta \ge 2$) each of whose arcs goes from $D_i$ to $D_{i+1}$ for some $i \in \{1, 2, \ldots, \eta-1\}$. Suppose that, in the CS-graph $PT_D^{(\eta)}$ of $D$, $\kappa_\eta$ divides $\kappa_p$ and $i-j\equiv i'-j'\pmod{\kappa_\eta}$ for any $({i},{j})$, $({i'},{j'})\in I(D_{p,p+1})$ for each $p=1$, $\ldots$, $\eta-1$. Then there exist an $(x,k)$-path and a $(y,k)$-path for some $k \in \ZZ_{\kappa_\eta}$ in $PT_D^{(\eta)}$ if and only if  $x \equiv y \pmod{\kappa_\eta}$ for any $x$, $y$ in $\ZZ_{\kappa_r}$ and for any $r=1$, $\ldots$, $\eta$.
\label{lem:char}
\end{Lem}
\begin{proof}
To show the `only if' part, suppose that, for some $r \in \{1, \ldots, \eta \}$, there exist an $(x,k)$-path and a $(y,k)$-path in $PT_D^{(\eta)}$ for some $k \in \ZZ_{\kappa_\eta}$ and $x$, $y$ in $\ZZ_{\kappa_r}$. If $r=\eta$, then $x \equiv y \pmod{\kappa_\eta}$ and the lemma is trivially true. Now suppose that $r < \eta$. Then there exist an $(x,k)$-path $xs_{r+1}\cdots s_{\eta-1}k$ and a $(y,k)$-path $yt_{r+1}\cdots t_{\eta-1}k$ where  $\{s_m, t_m\} \subset \ZZ_{\kappa_m}$ for $m=r+1$, $\ldots$, $\eta-1$. By the definition of $B_D$,  for each $m=r$, $\ldots$, $\eta-1$,
\[(s_{m}-s_{m+1})-(g_{m}-g_{m+1})-1 \in \langle \kappa_{m},\kappa_{m+1} \rangle \subset \langle \kappa_\eta \rangle\]
and
\[(t_{m}-t_{m+1})-(h_{m}-h_{m+1})-1\in \langle \kappa_{m},\kappa_{m+1} \rangle \subset \langle \kappa_\eta \rangle\]
where $s_r=x$, $t_r=y$, $s_\eta=t_\eta=k$, $(g_{m},g_{m+1}) \in I(D_{m,m+1})$,  $(h_{m},h_{m+1}) \in I(D_{m,m+1})$, $\langle \kappa_\eta \rangle=\{\alpha \kappa_\eta \mid \alpha \in \ZZ\}$, and $\langle \kappa_{m},\kappa_{m+1} \rangle=\{\beta\kappa_{m}+\gamma\kappa_{m+1} \mid \beta \in \ZZ, \gamma \in \ZZ\}$. By the hypothesis, $(g_m-g_{m+1})-(h_m-h_{m+1}) \in \langle \kappa_\eta \rangle$ and so $(s_m-s_{m+1})-(t_m-t_{m+1}) \in \langle \kappa_\eta \rangle$ for each $m=r$, $\ldots$, $\eta-1$. Therefore, $(x-k)-(y-k) \in \langle \kappa_\eta \rangle$ and this completes the proof of the `only if' part.

We prove the `if' part by induction on $\eta-p$ for $p=1$, $\ldots$, $\eta-1$. Suppose that $x \equiv y \pmod{\kappa_\eta}$ for $x$, $y$ in $\ZZ_{\kappa_{\eta-1}}$. Since $D$ is linearly connected, there is an element $(k,l)$ in $I(D_{\eta-1,\eta})$. Since $x$, $y$, and $k$ belong to $\ZZ_{\kappa_{\eta-1}}$, $x \equiv k+m+1 \pmod{\kappa_{\eta-1}}$  and $y \equiv k+m'+1 \pmod{\kappa_{\eta-1}}$ for some integers $m$ and $m'$.  Then, by definition of $B_D$, $(x,z) \in B_{\eta-1,\eta}$ and $(x,z') \in B_{\eta-1,\eta}$ for a vertex $z \in \ZZ_{\kappa_\eta}$ satisfying $z \equiv l+m \pmod{\kappa_\eta}$ and  a vertex $z' \in \ZZ_{\kappa_\eta}$ satisfying $z' \equiv l+m' \pmod{\kappa_\eta}$. Since $\kappa_\eta \mid \kappa_{\eta-1}$, $x \equiv k+m+1 \pmod{\kappa_{\eta}}$ and $y \equiv k+m'+1 \pmod{\kappa_{\eta}}$. Since $x \equiv y \pmod{\kappa_\eta}$, $m \equiv m' \pmod{\kappa_\eta}$ and so $z \equiv z' \pmod{\kappa_\eta}$. Since both $z$ and $z'$ belong to $\ZZ_{\kappa_\eta}$, $z=z'$ by the definition of $PT_D^{(\eta)}$ and so the `if' part is true for $\eta-p=1$.

Suppose that the `if' part is true for $\eta-p$ $(p \ge 1)$. Then we assume that $x \equiv y \pmod{\kappa_\eta}$ for $x$, $y$ in $\ZZ_{\kappa_{\eta-p-1}}$. By applying a similar argument as above, we may show that there exist vertices $w$ and $w'$ in $\ZZ_{\kappa_{\eta-p}}$ such that $w \equiv w' \pmod{\kappa_\eta}$ and there are edges $(x,w)$ and $(y,w')$ in $PT_D^{(\eta)}$. Then, by the induction hypothesis, there exist a $(w,z)$-path $P$ and a $(w',z)$-path $Q$ for $z \in \ZZ_{\kappa_\eta}$. Accordingly $xP$ and $yQ$ are an $(x,z)$-path and a $(y,z)$-path, respectively, and this completes the proof of the `if' part.
\end{proof}

\begin{Thm} Let $D$ be a linearly connected digraph with only nontrivial strong components $D_1$, $D_2$, $\ldots$, $D_\eta$ ($\eta \ge 2$) each of whose arcs goes from $D_i$ to $D_{i+1}$ for some $i \in \{1,2 \ldots, \eta-1\}$.  Suppose that a graph $G$ is the limit of $\{C(D^m)\}_{m=1}^{\infty}$.
Then $G$ is a union of complete graphs if and only if $\kappa_\eta$ divides $\kappa_p$  and $i-j\equiv i'-j'\pmod{\kappa_\eta}$ for any $({i},{j})$, $({i'},{j'})\in I(D_{p,p+1})$ for each $p=1$, $\ldots$, $\eta-1$.
\label{thm:simplegeneral}
\end{Thm}
\begin{proof}
By induction on the number $\eta$ of strong components of a linearly connected digraph. By Theorem~\ref{thm:simple} and Lemma~\ref{lem:char}, the statement is true for $\eta=2$.  Suppose that the statement is true for $\eta-1$ ($\eta \ge 3$). Let $D$ be a linearly connected digraph with only nontrivial strong components $D_1$, $D_2$, $\ldots$, $D_\eta$ each of whose arcs goes from $D_i$ to $D_{i+1}$ for some $i \in \{1, 2, \ldots, \eta-1\}$. We delete $D_1$ from $D$ to obtain a linearly connected digraph $F$ with $\eta-1$ strong components $D_2$, $\ldots$, $D_\eta$. By the induction hypothesis, $\{C(F^m)\}_{m=1}^{\infty}$ converges to a union of complete graphs if and only if $\kappa(D_\eta)$ divides $\kappa(D_p)$  and $i-j\equiv i'-j'\pmod{\kappa(D_\eta)}$ for any $({i},{j})$, $({i'},{j'})\in I(D_{p,p+1})$ for each $p=2$, $\ldots$, $\eta-1$.

 By the hypothesis, $\{C(D^m)\}_{m=1}^{\infty}$ converges to $G$. By Theorem~\ref{thm:main3}, $G$ is an expansion of $PT_D^{(\eta)}$. Since there is no way to reach from a vertex in $D_j$ to a vertex in $D_i$ for $j > i$, $\{C(F^m)\}_{m=1}^{\infty}$ converges to a graph, say $G'$, and  $G'$ is an expansion of $PT_D^{(\eta)}-\ZZ_{\kappa_1}$.  Therefore  $G'=G-V(D_1)$ by the way in which the expansion is constructed in the proof of  Theorem~\ref{thm:main3}.

To show the `only if' part, suppose that $G$ is a union of complete subgraphs. Then, since $G'=G-V(D_1)$, $G'$ is a union complete subgraphs. Hence, by induction hypothesis, $\kappa_\eta$ divides $\kappa_p$ and $i-j\equiv i'-j'\pmod{\kappa_\eta}$ for any $({i},{j})$, $({i'},{j'})\in I(D_{p,p+1})$ for each $p=2$, $\ldots$, $\eta-1$.

We first make the following observation: Suppose that the complete graph replacing $x$ and the complete graph replacing $y$  are contained in the same component in $G$ for $x$ and $y$ in $\ZZ_{\kappa_2}$. Then, for some $r \in \{2,\ldots,\eta \}$, there are an $(x,k')$-path and a $(y,k')$-path for some $k' \in \ZZ_{\kappa_r}$ by the definition of expansion. Since $D$ is linearly connected, $PT_D^{(\eta)}$ is connected and so there is a $(k',k)$-path for some $k \in \ZZ_{\kappa_\eta}$ in $PT_D^{(\eta)}$. Then, by Lemma~\ref{lem:char}, $x\equiv y \pmod{\kappa_\eta}$.

Now we show that $\kappa_\eta \mid \kappa_1$. Since $D$ is linearly  connected, there is an edge $(a,b)$ in $B_{1,2}$ and so, by the definition of $B_D$, for some $(i,j) \in I(D_{1,2})$ and for some integer $l$,
\[a \equiv i+l+1 \pmod{\kappa_1}, \quad b \equiv j+l \pmod{\kappa_2}.\]
Since $(i,j) \in I(D_{1,2})$, it is true that $(i+l+\kappa_1+1, j+l+\kappa_1)$ is an edge of $B_{1,2}$. Since $i+l+\kappa_1+1 \equiv a \pmod{\kappa_1}$, it is true that $(a,j+l+\kappa_1)$ is an edge of $B_{1,2}$. Then, since $(a,b)$ is also an edge of $PT_D^{(\eta)}$, the complete graphs replacing $b$ and $j+l+\kappa_1$ should be contained in the same component in $G$ by the hypothesis. By the observation above, $b \equiv j+l+\kappa_1 \pmod{\kappa_\eta}$. Since $b \equiv j+l \pmod{\kappa_2}$ and $\kappa_\eta \mid \kappa_2$, we can conclude that $\kappa_\eta \mid \kappa_1$.

Take $({i},{j})$, $({i'},{j'})\in I(D_{1,2})$.
Without loss of generality, we may assume that $i'>i$.
Since $({i},{j})\in I(D_{1,2})$, by the definition of $B_D$, $(i+ (i'-i)+1,  j+(i'-i)) $ is an edge of $B_{1,2}$ and so is $(i'+1,j +i'-i )$.
In addition, $(i'+1,j')$ is an edge of $B_{1,2}$ as $({i'},{j'}) \in I(D_{1,2})$.
Therefore both $(i'+1,j+i'-i)$ and  $(i'+1,j')$ are edges of $B_{1,2}$. Then, by the hypothesis, the complete graphs replacing $j+i'-i$ and $j'$ should be contained in the same component in $G$.  By the above observation again, $j+i'-i \equiv j' \pmod{\kappa_\eta}$.
Thus $i-j\equiv i'-j'\pmod{\kappa_\eta}$.

To show the `if' part, suppose that $\kappa_\eta$ divides $\kappa_p$  and $i-j\equiv i'-j'\pmod{\kappa_\eta}$ for any $({i},{j})$, $({i'},{j'})\in I(D_{p,p+1})$ for each $p=1$, $\ldots$, $\eta-1$. By the induction hypothesis, $G'$ is a union of complete subgraphs. Take a vertex $a\in \ZZ_{\kappa_1}$ of $PT_D^{(\eta)}$. If $a$ has no neighbor in $B_{1,2}$, then its degree is zero and the complete graph representing it in $G$ is a component of $G$.
Suppose that $a$ has a neighbor in $B_{1,2}$.
Let $(a,b)$ and $(a,c)$ are edges of $B_{1,2}$. Then, by definition, for some $({i},{j})$, $({i'},{j'})\in I(D_{1,2})$ and
for some integers $l$, $\l'$,
\[\begin{array}{llll}
a \equiv i+\l+1 \pmod{\kappa_1},   & b \equiv j+ l \pmod{\kappa_2},\\
a \equiv i'+l'+1 \pmod{\kappa_1}, & c\equiv j'+ l' \pmod{\kappa_2}.
\end{array}\]
Since $\kappa_\eta|\kappa_1$,
\[ a \equiv i+l+1  \equiv i'+l'+1 \pmod{\kappa_\eta},\]
and so $l-l'\equiv i'-i\pmod{\kappa_\eta}$.
Since $\kappa_\eta|\kappa_2$, 
\[b-c\equiv (j+l) -(j'+l') \equiv (j-j')+(l-l') \equiv (j-j')+ (i'-i) \equiv 0 \pmod{\kappa_\eta}.\]
Thus, by Lemma~\ref{lem:char}, there exist a $(b,z)$-path and a $(c,z)$-path for some $z \in \ZZ_{\kappa_\eta}$.  By the definition of  expansion, the complete graph replacing $b$ and the complete graph replacing $c$ induce the complete subgraph of $G$. Thus $G$ is a union of complete graphs.
\end{proof}

We may translate Theorem~\ref{thm:simplegeneral} into matrix version as follows:
\begin{Cor}
Suppose that $A \in \BBB_n$ is a matrix in the form given in (\ref{form}) where the block $A_{ii}$ has order at least two for each $i=1$, $\ldots$, $\eta$. Let $A'$ be the limit of $\{\Gamma(A^m) \}_{m=1}^\infty$. Then the following are equivalent:
 \begin{itemize}
 \item[(a)] The matrix $PA'P^T$ is a JBD matrix.
 \item[(b)] For each $p=1$, $\ldots$, $\eta-1$, $\kappa_\eta$ divides $\kappa_p$  and $i-j\equiv i'-j'\pmod{\kappa_\eta}$ whenever $A_{ij}^{(p,p+1)}$ and $A_{i'j'}^{(p,p+1)}$ defined in $(\S)$ are nonzero matrices.
 \end{itemize}
\end{Cor}

\section{Closing remarks}
In this paper, we found the limit of the matrix sequence $\{\Gamma(A^m) \}_{m=1}^\infty$ for a matrix $A$ in the form given in (\ref{form}) when all of the diagonal blocks are of order at least two.  We know from Theorem~\ref{thm:trivial} that there are other cases where $\{\Gamma(A^m) \}_{m=1}^\infty$ converges. We suggest that the limit of $\{\Gamma(A^m) \}_{m=1}^\infty$ be computed for $A$ satisfying the condition for each of such cases.

\end{document}